\documentclass[a4paper,12pt]{extarticle}

\usepackage[T1,T2A]{fontenc}
\usepackage[utf8]{inputenc}
\usepackage[english]{babel}
\usepackage{hyphenat}
\usepackage[left=2cm,right=2cm, top=2cm,bottom=2cm,bindingoffset=0cm]{geometry}
\usepackage{amsmath}
\usepackage{amssymb}
\usepackage{amsfonts}
\usepackage{amsthm}

\usepackage{paralist}
\usepackage{enumitem}
\usepackage[colorlinks=false, urlcolor=blue]{hyperref}
\usepackage{perpage}
\usepackage{setspace}

\theoremstyle{plain}

\newtheorem{theorem}{Theorem}[section]

\newtheorem{lemma}[theorem]{Lemma}

\newtheorem{corollary}[theorem]{Corollary}

\newtheorem*{claim*}{Claim}
\newtheorem*{theorem*}{Theorem}
\newtheorem*{lemma*}{Lemma}
\newtheorem*{observation*}{Observation}
\newtheorem*{corollary*}{Corollary}

\theoremstyle{definition}

\newtheorem{definition}[theorem]{Definition}
\newtheorem{problem}[theorem]{Problem}

\newtheorem*{definition*}{Definition}
\newtheorem*{problem*}{Problem}
\newtheorem*{problems*}{Problems}
\newtheorem*{fact*}{Fact}

\theoremstyle{remark}

\newtheorem*{example*}{Example}
\newtheorem*{remark*}{Remark}

\newcommand{\EE}{\operatorname{\mathsf{E}}}

\newcommand{\diff}{\mathop{}\!\mathrm{d}}
\newcommand{\Uniform}{\mathrm{U}}
\newcommand{\Binomial}{\mathrm{Bin}}
\newcommand{\Pois}{\mathrm{Pois}}
\newcommand{\eps}{\varepsilon}
\newcommand{\eqdef}{\operatorname{\overset{def}{=}}}
\newcommand{\rk}{\operatorname{rk}}
\newcommand{\Mat}{\mathcal{M}}

\renewcommand{\Pr}{\mathsf{P}}
\renewcommand{\geq}{\geqslant}
\renewcommand{\leq}{\leqslant}
\renewcommand{\emptyset}{\varnothing}

\begin{document}
\setlength{\parskip}{0.2cm}
\begin{center}
    \Large
    On the random minimum edge-disjoint spanning trees problem\\
    \vspace*{1cm}
    \large
     Dmitry Shabanov\footnote{Moscow Institute of Physics and Technology, Laboratory of Combinatorial and Geometric Structures; HSE University, Faculty of Computer Science}, Nikita Zvonkov\footnote{HSE University, Faculty of Computer Science}\footnote{The work was supported by the HSE University Basic Research Program}\\
\end{center}

\section*{Abstract}

It is well known that finding extremal values and structures can be hard in weighted graphs. However, if the weights are random, this problem can become way easier. In this paper, we examine the minimal weight of a union of $k$ edge-disjoint trees in a complete graph with independent and identically distributed edge weights. The limit of this value (for a given distribution) is known for $k=1,2$. We extend these results and find the limit value for any $k>2$. We also prove a related result regarding the structure of sparse random graphs.

\section{Introduction}

Let us recall some definitions. A spanning tree in a graph is a set of edges that form a tree and covers all the vertices. For a weighted graph, by its weight denote the sum of its edge weights. The minimum spanning tree (MST) in a weighted graph (i.e. every edge has it's "weight" being some real non-negative value) is the spanning tree with the minimal sum of its edge weights.

The set of edges $T_1\cup\dots\cup T_k$ is said to be a union of $k$ edge-disjoint spanning trees in a graph if $T_i$ forms a spanning tree and $T_i\cap T_j=\varnothing$ for any $i\neq j$. In the paper we are interested in the minimal weight of a union of $k$ edge-disjoint spanning trees.

In the classic paper \cite{F1985} A. Frieze considered a so-called Random MST problem:
\begin{problem}\label{prob:mst}
    Consider a complete weighted graph on $n$ vertices. Let every edge have a random weight, distributed as $\Uniform[0,1]$ (uniformly on [0,1]). Let all the weights be a set of independent values. Then, what is the limit weight of the MST of this graph?
\end{problem}

He figured out that in this case the limit value is $\zeta(3)$, and, in general, for any distribution $X$ such that for $\xi\sim X$, $\lim\limits_{\eps\to0+}\Pr[\xi<\eps]/\eps=a$ and $\EE \xi^2<\infty$, if the edge weights are distributed as $X$, then the limit value for MST is $\displaystyle\frac{\zeta(3)}{a}$.

Frieze and Johansson in \cite{FRIEZE2017} solved the minimum $k$ edge--disjoint spanning trees problem for $k=2$ and found the asymptotic weight for large $k$, using $k$-cores of random graphs.

This work comprises two results: the first one improves on the efforts of Frieze and Johansson, solving the minimum $k$ edge--disjoint spanning trees problem for any $k>2$. The second result suggests a hierarchy of dense structures in a sparse random graph.

For any integer $k>2$, define $c'_k$ as the root of an equation on $c$, $c>0$:
\[c\times\frac{\sum\limits_{i\geq k-1}(2c)^i/i!}{\sum\limits_{i\geq k}(2c)^i/i!}=k.\]
\begin{theorem}\label{th:main}
    Consider a complete weighted graph $K_n$ on $n$ vertices with the weights being independent and identically distributed according to distribution $X$. Let $X\geq0$ a.s., $\Pr [X\leq \eps]\sim a\cdot\eps$ for $\eps\to0$. Then for any $k>1$, the minimal weight of a union of $k$ edge-disjoint spanning trees in $K_n$ converges in probability to
    \[\frac{1}{2a}\int\limits_{0}^{+\infty}x(1-\beta^2_k(x))\diff x,\]
    where $\beta_k(c)$ is 0 if $c<2c'_k$ and otherwise it is the largest root of the equation
        \[\beta=1-\Pr[\Pois(\beta c)\geq k],\]
    and $\Pois(\lambda)$ is the Poisson distribution with mean value $\lambda$.
\end{theorem}

In order to formulate the next result we need to define what a random graph is. There are two main models, which are in some way equivalent to each other.

Fix $n>0$ and consider a random permutation of edges of a complete graph on $n$ vertices: $\displaystyle (e_i)_{i\in\binom{n}{2}}$. By $G_i$ denote $(V, \{e_1,\dots,e_i\})$. Then $G_1\subsetneq G_2\subsetneq\dots\subsetneq G_{\binom{n}{2}}$ is called a \textit{random graph process}. By $G(n,m)$ we denote the distribution of $G_m$ in this process, this model is called \textit{uniform}. There is also an other model of random graphs called \textit{binomial}, denoted by $G(n,p)$ --- that is, every pair of vertices is connected by an edge independently with probability of $p$. Note the following: if some monotone graph property (i.e. either monotonously increasing, a property which cannot be lost after adding an edge to the graph, or monotonously increasing) is satisfied w.h.p. (i.e. with probability tending to $1$ as $n\to\infty$) in $G(n,m=cn)$ for any $c>c_0$, it's also satisfied w.h.p. in $G(n,p=2c/n)$ for any $c>c_0$. The reverse is also true. The proof can be found in \cite{Bollobas2001}.

\begin{definition}
    By a $k$-deeply connected component define a maximal by inclusion induced subgraph of $G$, containing $k$ edge--disjoint trees.
\end{definition}

\begin{theorem}\label{th:structure}
    Consider some $k\geq2$. There exists such $c'_k$ (defined later) that:
    \begin{enumerate}
        \item for any $c\in(0,2c'_k)$, w.h.p. no induced subgraph of $G(n, c/n)$ of size at least 2 contains $k$ edge--disjoint trees.
        \item for any $c>2c'_k$, w.h.p. there is a single non-trivial $k$-deeply connected component in $G(n, c/n)$ of size $\beta_k(c)n(1+o(1))$.
    \end{enumerate}
\end{theorem}

The remaining of the paper is structured as follows: at first we discuss so-called $k$-deep connectivity and its connection with the $k$-cores of graphs. Then, we obtain an algorithm that allows for the construction of the union of $k$ edge-disjoint trees with minimal weight in \textit{any} weighted graph and discuss its connection with the $k$-deeply connected components and the $k$-cores of a graph. Finally, we prove Theorems \ref{th:main} and \ref{th:structure} using the discussed structures and declare some of the possible further results.

\section{Structure of Sparse Random Graphs}

\subsection{$k$-deeply connected components}

Before we start discussing the algorithm for finding the union of $k$ edge-disjoint trees with minimal weight in the graph, we want to somewhat re-define connectivity. Traditionally, a graph is said to be connected if there is a path between any pair of vertices (or, equivalently, if some spanning tree can fit into it). That definition allows to split any graph into some \textit{connected components}, that is, the largest induced subgraphs which are connected.

Now, let us define a slightly different property.

\begin{definition}
    We say that a graph is \textit{$k$-deeply connected} if $k$ edge-disjoint spanning trees can fit into it.
\end{definition}

\begin{definition}
    A subset of vertices $C$ of a graph is said to form a \textit{$k$-deeply connected component} if:
    \begin{enumerate}
        \item its induced subgraph is $k$-deeply connected;
        \item for any $C'\supsetneq C$, $C'\ne C$, its induced subgraph is \textit{not} $k$-deeply connected.
    \end{enumerate}
\end{definition}

It's easy to see that 1-deep connectivity is the same as "traditional" connectivity. However, one of the natural properties of connectivity holds for any $k$.

\begin{lemma}
    For any $k$, any vertex of a graph can only belong to one $k$-deeply connected component, which means that any graph can be uniquely split into $k$-deeply connected components.
\end{lemma}

\begin{proof}
    Consider a vertex $v$ and the set $\mathcal{C}_v$ of all subsets of $V$ containing $v$, such that their induced graphs are $k$-deeply connected. It is easy to see that $\{v\}\in\mathcal{C}_v$, which, coupled with the fact than the size of sets in $\mathcal{C}_v$ is bounded by $|V|$, proves that there is \textit{at least one} $k$-deeply connected component, containing $v$.

    Now assume that there are two different components $C^1,C^2$ on the subsets $V^1, V^2$, both containing $v$ (which means that $C^1\cap C^2\neq\emptyset$). That means that there are $k$ edge-disjoint trees $T_1^1,T_2^1,\dots,T_k^1$ on $V^1$ and similarly, there are $T_1^2,T_2^2,\dots,T_k^2$ on $V^2$.

    Now for any $T_i^2$, consider $T'_i$, obtained by deleting all the edges of $T_i^2$, belonging to $C^1$. In this case $T^1_i\cup T'_i$ is connected on $V^1\cup V^2$ hence should contain some subtree $T^3_i$. For any $i\neq j$, $T_i^3$ and $T_j^3$ do not have common edges by construction which means that the graph induced by $V_1\cup V_2$ is $k$-deeply connected and therefore neither of $C^1, C^2$ is a connectivity component.
\end{proof}

\subsection{The structure of a deeply connected graph}

For any graph, by its $\kappa$-core denote the largest induced subgraph of minimal degree at least $\kappa$. Note that $\kappa$-core may be empty for $\kappa>1$. Here and later, for any graph $G$, we will denote its $\kappa$-core by $C_\kappa=(V(C_\kappa), E(C_\kappa))$.

\begin{lemma}\label{lem:representation}
    For any $k\geq2$, any $k$-deeply connected graph $G$ has a non-empty $(k+1)$-core.
\end{lemma}
\begin{proof}
    Note that the degrees of the vertices in $G$ are at least $k$ (there are $k$ edge-disjoint spanning trees passing through every vertex). Consider any vertex $v$ with degree equal exactly to $k$. Then, for any of these $k$ trees, $v$ is a leaf, which means, that we can delete this vertex and the graph will remain $k$-deeply connected. Let's do in until there are no vertices of degree at most $k$. If the graph turned out to be empty, do back 2 steps. In that moment, the graph only consisted of 2 vertices and at least $k$ edges. Since $k\geq 2$ we get the contradiction.
\end{proof}

Moreover, any $k$-deeply connected graph can be represented in the following way:
\begin{itemize}
    \item The first layer consists of all the vertices of the $k+1$-core;
    \item At layer $t$, all the vertices have exactly $k$ edges, connecting them to the lower levels and at least one edge, connecting them to a vertex at layer $t-1$.
\end{itemize}
Let's call this a \textit{normal representation} of a $k$-deeply connected graph.

\section{Matroids}

\subsection{Quick survey}\label{sec:matroid_survey}

All of the facts in this section can be found in most of the books on matroid theory (in particular, in \cite{W76}). Shortly, matroid is a structure that allows for a greedy algorithm to work properly due to its properties:

\begin{definition}
    A \textit{matroid} is any structure $\Mat=(E,\mathcal{I})$, where $E$ is a finite set, known as the \textit{ground set}, $\mathcal{I}\subseteq 2^E$, the elements of $\mathcal{I}$ are called \textit{independent}\footnote{Note that this definition has nothing to do with independent sets of vertices in the graph; in this paper we don't operate with the independent sets of vertices at all} sets and, moreover, the following conditions are satisfied:
    \begin{enumerate}
        \item For any independent $A$ and any $B\subseteq A$, $B$ is also independent;
        \item For any independent $A,B$ such that $|A|>|B|$, there exists $a\in A\backslash B$ such that $B\cup\{a\}$ is also independent.
    \end{enumerate}
\end{definition}

Any maximal by inclusion independent set is called a \textit{basis} of matroid. Note that the second property (also known as matroid property) implies that all the bases are of the same size.

For any matroid $\Mat=(E,\mathcal{I})$, by its \textit{rank function} we denote such $\rk:2^E\to \mathbb{Z}_+$ that maps $X$ to the maximal size of an independent subset of $X$. This function has a number of properties, one of which is quite important: for any $\Mat$ on the ground set $E$ with rank function $\rk$ and any $A,B\subset E$,
\[\rk A + \rk B \geq \rk(A\cup B) + \rk(A\cap B).\tag*{(RANK)}\label{eq:rank}\]
In particular, $\rk(A\cup B)\leq\rk A + \rk B$ for any $\Mat,A,B$.

The rank function proves to be very useful in solving optimisation problems on matroids. Consider a matroid $\Mat$ with additional \textit{weight function} $W:E\to\mathbb{R}$. Consider also a quite natural generalisation: for any $A\subseteq E$, $W(A):=\sum\limits_{a\in A}W(a)$. Then, optimisation problem of finding the lightest basis of a matroid can be solved by a greedy algorithm:
\begin{enumerate}
    \item Sort the elements of $E$ in ascending order of weight:
    \[E=(e_1,\dots,e_{|E|}),W(e_1)\leq\dots\leq W(e_{|E|}).\]
    \item Choose the smallest $i$ such that $\rk_\Mat(\{e_i\})=1$ and let $a_1=i$; if no such $i$ exists, terminate.
    \item If $a_1,\dots,a_{k-1}$ are determined, then choose the smallest $i$ such that $\rk_\Mat(\{e_{a_1},\dots,e_{a_{k-1}},e_i\})=k$ and let $a_k=i$; if no such $i$ exists, terminate.
    \item Output the set of all $e_{a_i}$. It is the lightest basis.
\end{enumerate}

For any number of matroids on the same ground set $\{\Mat_i=(E,\mathcal{I}_i)|i\in[k]\footnote{here and later by $[k]$ we denote the set $\{1,2,\dots,k\}$}\},$ by the union of matroids denote
\[\bigcup\limits_{i\in[k]}\Mat_i=(E,\mathcal{I}); A\in \mathcal{I} \textnormal{ iff } A=\bigsqcup\limits_{i\in[k]}A_i, \textnormal{ such that } A_i\in\mathcal{I}_i.\]

This structure turns out to also be a matroid, moreover, its rank function can be calculated based on $\rk_{\Mat_i}$:

\begin{lemma}\label{lem:union}
    For any $\Mat = \bigcup\limits_{i\in[k]}\Mat_i$, the following holds:
\[\rk_{\Mat}(X)=\min\limits_{U\subseteq X}\left(|X\backslash U|+\sum\limits_{i\in[\ell]}\rk_{\Mat_i}(U)\right).\]
\end{lemma}

By $K_n$ denote a complete graph on $n$ vertices. Note that \[\Mat=(E(K_n), \{E'\subset E|E'\textnormal{ doens't contain any cycles}\})\] is a matroid, known as \textit{a graphic matroid}. It is straightforward to see that $\rk E'= n - (\textnormal{\# of connectivity components in }E')$

\subsection{The rank of a graph}

Consider a matroid $\Mat^k=\bigcup\limits_{1\leq i\leq k} \Mat_i$, where all $\Mat_i$ are graphic matroids on the same ground set. So, any set of edges $E'$ is independent in $\Mat^k$ if and only if it can be split into $k$ forests. For a fixed $k$, by the rank of a graph $G$ we denote $\rk G:=\rk_{\Mat^k}(E(G))$, $\rk E := \rk_{\Mat^k} E$.

\begin{definition}
    For any matroid $\Mat=(E,\mathcal{I})$, $E'\subsetneq E$, $e\in E\backslash E'$ we say that $e$ is \textit{dependent} on $E'$ if $\rk_{\Mat}(E')=\rk_{\Mat}(E'\cup\{e\})$. Otherwise, it is said to be \textit{independent} of $E'$.
\end{definition}

We would also say that $e$ is dependent on $E$ when $\rk_{\Mat^k}(E)=\rk_{\Mat^k}(E\cup\{e\})$ without specifying the matroid.

\begin{lemma}\label{lem:independence}
    In a matroid, for any set $E'\subset E$, any maximal independent subset $I$ of $E'$ and any $e\not\in E'$, $e$ is independent of $E'$ if and only if it is independent of $I$.
\end{lemma}

\begin{proof}
    By definition $e$ is independent of $E'$ if and only if $\rk E' = \rk (E'\cup\{e\}) - 1$. Consider any maximal independent subset $I'\subset E'\cup\{e\}$. Since $|I'|=|I| + 1$, for some $x\in I'\backslash I$, $\rk I\cup\{x\}=\rk I'$. If $x\neq e$, $I$ is not maximal independent set of $E'$, which means that $x$ must be $e$, i.e. $e$ is independent of $I$. On the other hand, if $e$ is independent of $I$, then $\rk E'\cup\{e\}\geq\rk I\cup\{e\}=\rk E' + 1$, which means that $e$ is independent of $E'$.
\end{proof}

\begin{lemma}\label{lem:criteria_for_independence}
    Any set of edges $E$ is independent in $\Mat^k$ if and only if it does not contain subgraphs with $m$ vertices and more than $k(m-1)$ edges.
\end{lemma}

\begin{proof}
    Note that if there exists an induced subgraph on $m$ vertices containing at least $k(m-1)+1$ edges, then $E$ is not independent: let's split $E$ into $k$ forests. At least one of them contains not less than $m$ edges on $m$ vertices.

    On the other hand, $\rk_{\Mat^k}(E)=\min\limits_{E'\subset E}\Big(|E\backslash E'| + k\cdot\big(|V|-(\textnormal{\# of connectivity components in }E')\big)\Big)$. For every such $E'$, let us add to it all the edges that do not decrease the number of $k$-deeply connected components. Let $E''$ denote the resulting set. Now,
    \[\rk_{\Mat^k}(E)=\min\limits_{E'\subset E}\Big(|E\backslash E''(E')| + k\cdot\big(|V|-(\textnormal{\# of connectivity components in }E''(E'))\big)\Big).\]

    But this is equal to
    \[\rk_{\Mat^k}(E)=\min\limits_{\begin{smallmatrix}V_1\sqcup\dots\sqcup V_\ell = V\\ V_i\textnormal{ are connected in }E\end{smallmatrix}}\big((\textnormal{\# of edges between different }V_i) + k\cdot(|V(A)|-\ell)\big).\]
    For any $V_1\sqcup\dots\sqcup V_\ell = V(E)$, they contain at most $\sum_{i\in[\ell]}k(|V_i|-1)=k|V(E)|-k\ell=k(|V(E)|-\ell)$ edges, which means that
    \[(\textnormal{\# of edges between different }V_i) + k\cdot(|V(E)|-\ell)\geq|E|,\]
    but since $\rk_{\Mat^k}(E)\leq |E|$, we have $\rk_{\Mat^k}(E)=|E|$.
\end{proof}

\begin{lemma}\label{lem:connection}
    Consider $G=(V,E)$ and an edge $e\not\in E$. Then $e$ is dependent of $E$ if and only if the vertices of $e$ fall into the same $k$-deeply connected component of $G$.
\end{lemma}

\begin{proof}
    Let $V_1,V_2,\dots, V_\ell$ be the sets of vertices of $k$-deeply connected components of $G$. Let $(V_i,E_i)$ be the corresponding induced subgraphs and let $m_i=|E_i|,n_i=|V_i|$. Note that the rank of $E_i$ is equal to $k(n_i-1)$. Consider the following order on $E$:
    \[\begin{array}{c}
        \{e_1,e_2,\dots,e_{m_1}\}=E_1,\\
        \{e_{m_1+1},e_{m_1+2},\dots,e_{m_1+m_2}\}=E_2,\\
        \dots,\\
        \{e_{m_1+\dots+m_{\ell-1}+1},e_{m_1+\dots+m_{\ell-1}+2},\dots,e_{m_1+\dots+m_{\ell}}\}=E_\ell,\\
        \{e_{m_1+\dots+m_{\ell}+1},e_{m_1+\dots+m_{\ell}+2},\dots,e_{|E|}\}=E\backslash\bigcup\limits_{i\in[\ell]}E_i.
    \end{array}
    \]
    Consider a set $I=\{e_i|e_i\textnormal{ is independent of }\bigcup\limits_{j<i}\{e_j\}\}$. Then $I$ is the maximal independent set of $E$, moreover, since $\rk E_1\cup\dots\cup E_r=\sum\limits_{i\in[r]}k(n_i-1)$, $|I\cap E_i|=k(n_i-1)$ for any $i\in[\ell]$. That means that for any $V_i$, its induced subgraph in $(V,I)$ is also $k$-deeply connected. Also note that any $V'$ induces a $k$-deeply connected subgraph in $I$ if and only if this subgraph contains exactly $k(|V'|-1)$ edges.

    Now, consider $e\not\in E$. If $e$ is dependent of $E$ then Lemma \ref{lem:independence} implies that $I\cup\{e\}$ is not independent. Therefore by Lemma \ref{lem:criteria_for_independence}, there is a subgraph in $I\cup\{e\}$ containing $m$ vertices and more than $k(m-1)$ edges. But that means that before adding $e$ this subgraph had exactly $k(m-1)$ edges, i.e. it was $k$-deeply connected.
\end{proof}

Let us understand the significance of this lemma. As it was noted in Subsection \ref{sec:matroid_survey}, there is an algorithm that allows for finding the minimum-weight basis of any weighted matroid. What the lemma states is that the condition on adding $e_i$ to the minimum basis can be replaced with the following:
\[e_i \textnormal{ connects different $k$-deeply connected components in } (V,\{e_{a_1},\dots,e_{a_{k-1}}\}),\]

which is by Lemma \ref{lem:independence} equivalent to connecting different $k$-deeply connected components in
$(V,\{e_1,\dots,e_{i-1}\})$. This connection states the similarity between the famous Kruskal algorithm (which is described in \cite{kruskal1956shortest}) and this one. Basically, one is derived from the other by re-defining what connectivity is.

\subsection{The $k$-core of a random graph}

B. Pittel, J. Spencer and N. Wormald examined \cite{Pittel1996} the structure of the $k$-core in a random graph and acquired the following result.

\begin{definition}
    Let $\Pois(\lambda)$ denote a Poisson distribution with the expected value of $\lambda$.
\end{definition}

\begin{definition}
    By $\pi_k(c)$ denote $\Pr[\Pois(c)\geq k]$. By $f_k(c)$ denote $\sum\limits_{n=k}^{+\infty}c^n/n!=\pi_k(c)\cdot e^c.$
\end{definition}

\begin{theorem}[Pittel et al.]\label{lem:k_core_density}
    Let $\gamma_k$ denote $\gamma_k=\inf\left\{\lambda/\pi_{k-1}(\lambda)\right\}.$ Let $\lambda_k(c)$ denote the larger root of the equation $c=\lambda/\pi_{k-1}(\lambda)$ for $c>\gamma_k$.

    For fixed $c>0$, consider a random graph $G(n, c/n)$. If $c>\gamma_k$, then w.h.p. the size of the $k$-core of $G(n, c/n)$ is equal to $\pi_{k}(\lambda_k(c))n(1+o(1))$. If $c<\gamma_k$, then w.h.p. the $k$-core of $G(n, c/n)$ is empty.
\end{theorem}

For any graph $G=(E,V)$, by its \textit{density} we denote the average degree of a vertex in $V$. It is easy to see that the density of a graph is equal to $2|E|/|V|$.

\begin{theorem}[Pittel et al.]
    Fix $c>0$, let $\xi\sim\Pois(c)$. Let $E_k(c)$ denote $\EE[\xi|\xi\geq k]$. Then for $c>\gamma_k$,  w.h.p. the density of $k$-core of $G(n, c/n)$ is equal to $E_k(c)(1+o(1))$
\end{theorem}
Note that $E_k(c)= \frac{cf_{k-1}(c)}{f_k(c)}$. 

One very important way to think about the rank of a graph was suggested by A. Frieze et al. in \cite{FRIEZE2017}.

\begin{theorem}[Frieze et al.]\label{lem:rk}
    Consider a graph $G$. Then for any $k\geq2$,
    \[\rk_{\Mat^k} G=|E(G)|-|E(C_{k+1}(G))| + \rk_{\Mat^k} C_{k+1}(G),\]
    moreover, consider any largest independent subset of $E(G)$ in $\Mat^k$, denote it by $I$. Then for any edge $e\in E(G)$, if it is not in the $C_{k+1}(G)$, then $e\in I$.
\end{theorem}

\subsection{The rank of a random graph}

\begin{lemma}\label{lem:small_dense}
    For any $\eps, c>0$ there exists $\delta(\eps, c)>0$, such that w.h.p. there are no any subgraphs with density at least $2+\eps$ containing at most $\delta n$ vertices in $G(n,\lceil cn\rceil)$\footnote{Later we will drop the $\lceil~~\rceil$ brackets. Nevertheless, we should bear in mind that the number of edges in a graph is always integer.}.
\end{lemma}

Lemma \ref{lem:small_dense} is proven in the appendix.

\begin{corollary}
    For any $k>2$ and $c>0$, w.h.p. all the $k$-deeply connected components in the $G(n, cn)$ are either single vertices or of linear size.
\end{corollary}
\begin{proof}
    Every non-trivial $k$-deeply connected component has density at least 3 (that is the density of $K_4$, the smallest possible $k$-deeply connected component). Lemma \ref{lem:small_dense} implies that we can estimate the size of any $k$-deeply connected component from below by $\delta(1, c)n$.
\end{proof}

\begin{corollary}
    For any $k>2$ and $c>0$, w.h.p. the number of non-trivial components is $O_{c}(1)$.
\end{corollary}
\begin{proof}
    It follows from the previous two statements that w.h.p. this number is bounded by $1/\delta(1,c)$.
\end{proof}

As we remember from Theorem \ref{lem:k_core_density}, the density of a $k$-core is concentrated around some explicitly defined constant depending on $c$.

\begin{definition}
    Let $c'_k$ denote the moment when the limit density of the $k+1$-core in $G(n,c'_kn)$ equals $2k$.
\end{definition}

The following result was achieved by Gao, P\'erez--Gim\'enez and Sato in \cite{Gao2017}:
\begin{lemma}[Gao et al.]\label{lem:gao_prep}
    Fix $k\geq2$. Then, for any $\eps>0$, w.h.p. there is no subgraphs of average density at least $2k+\eps$ in $G(n,c'_kn)$.
\end{lemma}

From this lemma we can deduce the important result regarding the rank of the graph.
\begin{lemma}\label{lem:gao}
    Fix $k\geq2$. Then, for any $\eps>0$ with $G\sim G(n,c'_kn)$ and $C_{k+1}:=C_{k+1}(G)$ the following holds w.h.p.:
    \[\frac{k|V(C_{k+1})|-\rk C_{k+1}}{|V(C_{k+1})|}<\eps.\]
\end{lemma}
\begin{proof}
    Let $|V(C_{k+1})|=\alpha n(1+o(1))$ for some $0<\alpha <1$, $\alpha=\alpha(k)$.
    By Lemma \ref{lem:gao_prep}, the average density of any subgraph of $G$ is at most $2k+\eps$. Consider $A_1,\dots,A_r\subset V(G)$ --- the set of nontrivial (i.e. containing at least 2 vertices) $k$-deeply connected components of $G$. Lemma \ref{lem:small_dense} implies that $|A_i|\geq\delta n$ for any $i\in[r]$ and some $\delta=\delta(\eps,c'_k)$. Since $\sum|A_i|\leq n$ we have $r\leq 1/\delta$. For any  $A_i$, consider $T^{(i)}_1\cup\dots\cup T^{(i)}_k$ --- the corresponding non-intersecting spanning trees. Consider the largest independent (in terms of $\Mat^{(k)}$) subset $E'\subset E(G)$, containing all the trees $T^{(i)}_j$ that we have just mentioned. Then $E'$ also contains all the edges connecting different components, whereas any edge of $A_i$, not present in any of $T^{(i)}_j$, cannot be contained in $E'$ (otherwise $E'\cap A_i$ would consist of at least $k(m-1)+1$ edges on $m$ vertices).

    Further, $\rk G = |E'|$ and all the edges in $E\backslash E'$ lie in some of the $A_i$-s. Since all the $A_i$-s have density of at most $2k+\eps$, each of them contains at most $k+\eps|A_i|/2$ additional edges, and therefore
    \[|E\backslash E'| \leq \sum\limits_{i=1}^r (k+\eps|A_i|/2)\leq kr + \eps n/2\leq k/\delta(\eps,c'_k)+\eps n/2.\]

    By Theorem \ref{lem:rk},
    \begin{multline*}
        \rk C_{k+1}=
        \rk G-|E|+|E(C_{k+1})|=
        |E'|-|E|+|E(C_{k+1})|>\\
        >|E(C_{k+1})|-\eps n/2-O(1)>
        k|V(C_{k+1})|-\eps n,
    \end{multline*}
    consequently,
    \[\frac{k|V(C_{k+1})|-\rk C_{k+1}}{|V(C_{k+1})|}<\frac{k|V(C_{k+1})|-k|V(C_{k+1})|+\eps n}{n\alpha/2}<2\eps/\alpha.\]
    But since $2/\alpha$ is a constant for given $k$, by repeating the same argument with $\eps'=\alpha\eps/2$ we get the accurate estimate.
\end{proof}

\begin{lemma}\label{lem:core_rank}
    Consider a graph $G$ with non-empty $(k+1)$-core, denoted by $C_{k+1}$. Draw an additional edge $e$ in $G$. Assume that after drawing the edge, the $k+1$-core increased by $V_1$ vertices and $\rk C_{k+1}$ increased by $\Delta\rk C_{k+1}$. Let the subgraph of $G\cup\{e\}$ induced by these $V_1$ vertices contain $E_1$ edges, and let $E'$ be the number of edges between these vertices and the vertices of the old $(k+1)$-core. Then, if $kV_1-\Delta\rk C_{k+1}>0$, then
    \[E_1-V_1\geq kV_1-\Delta\rk C_{k+1}.\]
\end{lemma}
The proof of Lemma \ref{lem:core_rank} is given in the appendix.

\begin{lemma}
    For any $\eps>0$, $k\geq2$ and $c>c'_k$ by $C_{k+1}$ denote the $(k+1)$-core in $G\sim G(n,cn)$. Then w.h.p. 
    \[\frac{k|V(C_{k+1})|-\rk C_{k+1}}{|V(C_{k+1})|}<\eps.\]
\end{lemma}
\begin{proof}
    Fix $c'>0, \eps'>0$. By Lemma \ref{lem:gao}, w.h.p. the above statement holds for $c=c'_k$ and $\eps=\eps'/2$. Let's add the edges to $G(n,c'_kn/2)$ one by one until we get to $G(n, cn)$. Examine the value of $k|C_{k+1}|-\rk C_{k+1}$ throughout this process.

    In order to get from $G(n, c'_kn/2)$ to $G(n, cn)$, we need to draw $\left(c-\frac{c'_k}{2}\right)n$ additional edges. W.h.p. the size of $(k+1)$-core in $G(n, cn)$ lies in the interval $[0.99 c_0n, 1.01 c_0n]$ for some $c_0$. If throughout these additions the overall value of $k|C_{k+1}|-\rk C_{k+1}$ won't increase by $\frac{\eps'}{3}c_0n$, then
    \[\frac{k|C_{k+1}|-\rk C_{k+1}}{|C_{k+1}|}<\frac{\eps'}{2}+\frac{\eps'c_0n}{3|C_{k+1}|}<\eps'.\]
    Let $\eps_0=\displaystyle\frac{\eps'c_0}{6\left(c-\frac{c'_k}{2}\right)}$ and $\delta_0$ be a constant from Lemma \ref{lem:small_dense}, such that w.h.p. there are no subgraphs in $G(n,cn)$ of size at most $\delta_0n$ with $E/V\geq1+\eps_0$.

    Recall that for $t>c'_k$ there exists $\gamma_k(t)$ such that the size of $k$-core in $G(n,tn)$ converges to $\gamma_k(t)n(1+o(1))$ and this $\gamma$ is continuous and increasing. Fix $\eps_1$ such that for any $c'+k/2<t<c$, $\gamma_{k+1}(t+\eps_1)-\gamma_{k+1}(t)<\delta_0/3$ and let $\displaystyle S =\left\lceil\frac{c-\frac{c'_k}{2}}{\eps_1}\right\rceil, s=\left(c-\frac{c'_k}{2}\right)/S.$ W.h.p., for any $i<S$, the difference in sizes of $(k+1)$-cores in $G(n,(c'_k/2+is)n)$ and $G(n,(c'_k/2+(i+1)s)n)$ is less than $\delta_0n$, which means that for any $V_1$ vertices added simultaneously after drawing an additional edge, $V_1<\delta_0(n)$ and hence $E_1-V_1<V_1\cdot\eps_0$. Lemma \ref{lem:core_rank} implies that the difference $k|C_{k+1}|-\rk C_{k+1}$ can be bounded from above by $2\eps_0\cdot\left(c-\frac{c'_k}{2}\right)n=\frac{\eps'c_0n}{3}$.
\end{proof}

\begin{corollary}\label{cor:r_def}
    Fix any $c>0$. Then w.h.p.:
    \begin{enumerate}
        \item if $c<c_k$, $\rk G(n, cn)=cn(1+o(1))$;
        \item if $c>c_k$, $\rk G(n, cn)=r_k(c)n(1+o(1))$,
    \end{enumerate}
    where $\displaystyle r_k(c)=c-\frac{\lambda\pi_{k}(\lambda)}{2}+k\pi_{k+1}(\lambda), \lambda=\lambda_{k+1}(2c)$.
\end{corollary}

\begin{lemma}\label{lem:beta}
    By $\beta_k(c)$ denote $\lambda_{k+1}(c)/c$. Then $\displaystyle \frac{\partial}{\partial c}r_k(c)=1-\beta_k(2c)^2$. Additionally, $\beta_k(c)$ is the largest root of the following equation on $\beta$:
    \[\beta=\Pr[\Pois(\beta c)\geq k].\]
\end{lemma}

Proof of Lemma \ref{lem:beta} is given in the appendix.

\begin{lemma}\label{lem:exp}
    Fix any $n, m<\binom{n}{2}$. Then, for any $k>1$, with probability $1-O(n^{-1/6})$, the following holds: $(kn-\rk G(n,m))/n<f(m/n) + 3kn^{-1/5}$ for some exponentially decreasing $f$: $f(x)<e^{-bx}$ for some $b>0$.
\end{lemma}
\begin{proof}
    Enumerate the edges of $G\sim G(n,m)$ randomly (every numbering has the same probability to be chosen): $E(G)=(e_1,\dots,e_m)$. Then construct the following edge-disjoint graphs:
    \[G_1 = \left(V, \left\{e_i\in E(G)\Big| k|(i+1)\right\}\right), \dots, G_k = \left(V, \left\{e_i\in E(G)\Big| k|(i+k)\right\}\right).\]
    Any of these $k$ graphs distributed as a uniform random graph $G(n,\lceil m/k\rceil)$ or $G(n,\lfloor m/k\rfloor)$.

    As it was shown in \cite{F1985}, there exists an exponentially decreasing function $g$ s. t. the number of components (in the traditional meaning, not $k$-deep) in $G(n,m)$ w.h.p. is at most $ng(m/n)+3n^{4/5}$. Note that $\rk_\Mat(G)=|V(G)|-\mathcal{C}(G)$, where $\mathcal{C}(G)$ is the number of components in $G$ and $\Mat$ is a graphic matroid. Hence, w.h.p. for any $i\leq k$, $\rk_\Mat(G_i)\geq n(1-g(\lfloor m/k\rfloor/n))-3n^{4/5}$, which, in turn, means, that $\rk G\geq kn(1-g(\lfloor m/k\rfloor/n))-3kn^{4/5}$.
\end{proof}

\section{The Main Results}

\subsection{Proof of Theorem \ref{th:main}}

By Lemma \ref{lem:beta}, $\frac{\partial}{\partial c}r_k(c) = 1-\beta^2_k(2c)$ for any $c>0$. Moreover, $\beta^2_k$ is uniformly continuous on $(0,2c'_k)$ and on $(2c'_k,\infty)$. We will use this information to approximate the weight of the minimum union of $k$ edge-disjoint trees.

The weight of this union is equal to $\displaystyle\sum\limits_{i=1}^{\binom{n}{2}}\mathbb{I}_i\cdot w_i$, where $w_1\leq\dots\leq w_{\binom{n}{2}}$ are weights of edges $e_1,\dots,e_{\binom{n}{2}}$ and
\[\mathbb{I}_i=\mathbb{I}\Big[\rk\{e_1,\dots,e_{i+1}\}>\rk\{e_1,\dots,e_i\}\Big].\]

Let $W$ denote $W=\displaystyle\sum\limits_{i=1}^{\binom{n}{2}}\mathbb{I}_i\cdot w_i$; also, let $W_k$ denote $W_k\displaystyle\sum\limits_{i=1}^{k}\mathbb{I}_i\cdot w_i$. For any $C,\eps>0$ we can choose small $\delta>0$ such that for any $c\in(0,2C-\delta)$ such that $2c'_k\not\in(2c, 2c+\delta)$ the following holds:
\[\sup\limits_{x\in(2c, 2c+\delta)}|\beta^2_k(2x)-\beta^2_k(2c)|<\eps.\]

\renewcommand*{\thefootnote}{\fnsymbol{footnote}}
Then, it is possible to choose $0=c_1<c_2<\dots<c_m=C$ such that $c'_k\in\{c_2,\dots, c_{m}\}$ or $c'_k>C$, and $|\beta^2_k(2c_i)- \beta^2_k(2c_{i+1})|\leq\eps$\footnote[1]{Of course, we should pay extra attention to the case of $c_i=c'_k$. In that case, we want such $c_{i+1}$ to be chosen that $\displaystyle\left|\lim\limits_{x\to c_i+0}\beta^2_k(2x)- \beta^2_k(2c_{i+1})\right|<\eps$},$\eps/2<|2c_i-2c_{i+1}|<\eps$. W.h.p., the following two conditions are satisfied for any $i$:
\[\left|\frac{\rk G(n,c_in)}{n}-r_k(c_i)\right|<\eps^2;\tag*{(i)}\label{eq:cond_i}\]
\[|nw_{c_in}-2c_i/a|<\eps.\textnormal{ }\left(w_{c_in}\textnormal{ is approximately }\frac{2c_i}{an}\right)\tag*{(ii)}\label{eq:cond_ii}\]

Therefore, we could approximate $W_{Cn}$:
\renewcommand*{\thefootnote}{\arabic{footnote}}

\begin{multline*}
    \left|W_{Cn}-\sum\limits_{i=2}^{m}\frac{2c_i}{an}\cdot\left(1-\beta^2_k(2c_i)\right)n(c_i-c_{i-1})\right|<\\
    <\left|W_{Cn}-\sum\limits_{i=2}^{m}w_{c_in}(\rk G(n,c_in)-\rk G(n,c_{i-1}n))\right|+\\
    +\left|\sum\limits_{i=2}^{m}w_{c_in}(\rk G(n,c_in)-\rk G(n,c_{i-1}n))-\sum\limits_{i=2}^{m}\frac{2c_i\left(1-\beta^2_k(2c_i)\right)(c_i-c_{i-1})}{a}\right|\overset{\mathrm{(i)}}{<}\\
    \overset{\mathrm{(i)}}{<}Cn\cdot\max\limits_{i}{|w_{c_in}-w_{c_{i-1}n}|}+
    \left|\sum\limits_{i=2}^{m}nw_{c_in}(r_k(c_i)-r_k(c_{i-1}))-\sum\limits_{i=2}^{m}\frac{2c_i\left(1-\beta^2_k(2c_i)\right)(c_i-c_{i-1})}{a}\right|+\\
    +\eps^2\sum\limits_{i=2}^{m}nw_{c_in}\overset{\mathrm{(ii)}}{<}
    \frac{2C}{a}\cdot(\max\limits_{i}{|c_i-c_{i-1}|}+2\eps)+\\
    +\left|\sum\limits_{i=2}^{m}\frac{2c_i}{a}(r(c_i)-r(c_{i-1}))-\sum\limits_{i=2}^{m}\frac{2c_i\left(1-\beta^2_k(2c_i)\right)(c_i-c_{i-1})}{a}\right|+
    \eps\sum\limits_{i=2}^{m}(r_k(c_i)-r_k(c_{i-1}))+\\
    +\eps^2\sum\limits_{i=2}^{m}(\frac{2c_i}{a}+\eps)<
    \eps\left(\frac{6C}{a}+k+m\eps^2+\frac{2Cm\eps}{a}\right)+\\
    +\frac{2C}{a}\left|\sum\limits_{i=2}^{m}\left(r_k(c_i)-r_k(c_{i-1})-\left(1-\beta^2_k(2c_i)\right)(c_i-c_{i-1})\right)\right|.
\end{multline*}
Since $r_k(c_i)-r_k(c_{i-1})$ is equal to $\int\limits_{c_{i-1}}^{c_i}1-\beta^2_k(2t)\diff t$, $\left|\sum\limits_{i=2}^{m}\left(r_k(c_i)-r_k(c_{i-1})-\left(1-\beta^2_k(2c_i)\right)(c_i-c_{i-1})\right)\right|$ can be bounded by $C\cdot\max\limits_{i, c\in[c_{i-1}, c_i]}|\beta^2(2c)-\beta^2(2c_i)|=C\eps$. From that we conclude that
\[\left|W_{Cn}-\sum\limits_{i=2}^{m}\frac{2c_i}{a}\cdot\left(1-\beta^2_k(2c_i)\right)(c_i-c_{i-1})\right|<
\eps\left(\frac{6C}{a}+k+m\eps^2+\frac{2Cm\eps}{a}+C\right);\]
since $|c_i-c_{i-1}|>\eps/4$, $m\eps<4C$, which means that for any fixed $a,k,C$, the difference between $W_{Cn}$ and our approximation tends to 0 as $\eps\to 0$.

As $\eps$ tends to 0, the sum $\sum\limits_{i=2}^{m}\frac{2c_i}{a}\cdot\left(1-\beta^2_k(2c_i)\right)(c_i-c_{i-1})$ tends to
\[\int\limits_{0}^{C}\frac{2x}{a}(1-\beta_k^2(2x))\diff x=\frac{1}{2a}\int\limits_{0}^{2C}x(1-\beta_k^2(x))\diff x,\]
Therefore, for any $a,k,C$, \[W_{Cn}\overset{\mathrm{P}}{\to}\frac{1}{2a}\int\limits_{0}^{2C}x(1-\beta_k^2(x))\diff x.\]

\renewcommand*{\thefootnote}{\fnsymbol{footnote}}
Now, if we will manage to show that $\lim\limits_{n\to\infty}$\footnote[1]{here, by $\lim\limits_{n\to\infty}\xi_n$ we denote such $\xi$ that $\xi_n\overset{\mathrm{P}}{\to}\xi$.} $\sum\limits_{i=Cn}^{\binom{n}{2}}\mathbb{I}_i\cdot w_i$ tends in probability to 0 as $C$ tends to infinity, we will prove that the weight converges to
\begin{multline*}
    \lim\limits_{n\to\infty}\sum\limits_{i=1}^{\binom{n}{2}}\mathbb{I}_i\cdot w_i=
    \lim\limits_{C\to\infty}\left(\lim\limits_{n\to\infty} \sum\limits_{i=1}^{Cn}\mathbb{I}_i\cdot w_i+\lim\limits_{n\to\infty}\sum\limits_{i=Cn}^{\binom{n}{2}}\mathbb{I}_i\cdot w_i\right)=\\
    =\lim_{C\to\infty}\frac{1}{2a}\int\limits_{0}^{2C}x(1-\beta_k^2(x))\diff x+0=
    \frac{1}{2a}\int\limits_{0}^{\infty}x(1-\beta_k^2(x))\diff x.
\end{multline*}
\renewcommand*{\thefootnote}{\arabic{footnote}}

Note that w.h.p. $G(n,2kn\log n)$ is $k$-deeply connected, as its edges can be represented as a disjoint union of edges of $k$ different $G(n,2n\log n)$, each of which is w.h.p. connected. Hence, w.h.p.
\[\sum\limits_{i=1}^{\binom{n}{2}}\mathbb{I}_i\cdot w_i=\sum\limits_{i=1}^{2kn\log n}\mathbb{I}_i\cdot w_i \textnormal{ and } \rk G(n,2kn\log n) = k(n-1).\]

For any fixed $\eps>0$, there is sufficiently large $n$ that for any $n\leq m\leq 2kn\log n$ we can easily bound

\begin{multline*}
    X := \Pr\left[w_m-2m/an^2>\eps/n\right]=
    \Pr\left[\textnormal{(number of edges with }w_i-2m/an^2\leq\eps/n)<m\right]=\\=
    \Pr\left[\Binomial(n(n-1)/2,2m/n^2+a\eps/n)<m\right]<\Pr\left[\Binomial(n^2/2,2m/n^2+a\eps/2n)<m\right].
\end{multline*}
By using the Chernoff bound we obtain
\[X\leq\exp\left(-\frac{\left(\frac{na\eps}{4}\right)^2}{2m+\frac{3na\eps}{4}}\right)\leq \exp\left(-\frac{n^2}{m}\left(\frac{a^2\eps^2}{32+12a\eps}\right)\right)\leq e^{-
\sqrt{n}}.\]

Such small probability guarantees that for any fixed $\eps> 0$, w.h.p. $w_m\leq2m/an^2+\eps/n$ for any $n\leq m\leq 2kn\log n$.

Now we are ready to bound the total weight of the heavy edges. Consider a graph with edges $e_1,\dots,e_{\binom{n}{2}}$ and their weights $w_1<\dots<w_{\binom{n}{2}}$. Note that a graph formed by edges $e_1,\dots,e_m$ is distributed as $G(n,m)$. Consider $C\in\mathbb{N}, C>1$, then w.h.p.
\begin{multline*}
    \sum\limits_{i=Cn}^{\binom{n}{2}}\mathbb{I}_i\cdot w_i=\sum\limits_{i=Cn}^{2kn\log n}\mathbb{I}_i\cdot w_i \leq \sum\limits_{i=C}^{\lceil 2k\log n\rceil} \big(\rk G(n,(i+1)n)-\rk G(n,in)\big)w_{(i+1)n} \leq \\
    \leq \sum\limits_{i=C}^{\lceil 2k\log n\rceil} \Big(\big(k(n-1)-\rk G(n,in)\big)-\big(k(n-1)-\rk G(n,(i+1)n)\big)\Big)(2(i+1)/an+\eps/n) = \\
    = \big(k(n-1)-\rk G(n,Cn)\big)\left(\frac{2(C+1)+\eps a}{an}\right)+\frac{2}{an}\sum\limits_{i=C+1}^{\lceil 2k\log n\rceil-1} \big(k(n-1)-\rk G(n,in)\big)+0 \leq \\
    \leq (nf(C)+3kn^{4/5})\left(\frac{2(C+1)+\eps a}{an}\right)+\frac{2}{an}\sum\limits_{i=C+1}^{\lceil 2k\log n\rceil-1} nf(i)+3kn^{4/5} = \\
    = \frac{f(C)(2(C+1)+\eps a)}{a}+\frac{2\sum f(i)}{a}+o(1) \leq \frac{e^{-bC}(2(C+1)+\eps a)}{a}+\frac{2e^{-b(C+1)}}{a(1-e^{-b})}+o(1)
\end{multline*}
for some $b$ from the statement of Lemma \ref{lem:exp}. It is apparent that the sum of weights of heavy vertices in the minimum union of $k$ edge-disjoint trees converges to 0 as $C$ tends to $\infty$.

\subsection{Proof of Theorem \ref{th:structure}}
Fix any $C>0$ and consider a random graph process $G_1\subset G_2\subset\dots\subset G_{Cn}$, we know that $G_m\sim G(n,m)$. By Lemma \ref{lem:small_dense}, any $k$-deeply connected component in $G_{Cn}$ has density at least $3$ and its size should be at least $\eps(C)n$ for some function $\eps$. Let us fix $\delta$ such that for any $c$ such that $(c,c+\delta)\not\ni c'_k$, $\beta^2_k(2(c+\delta))-\beta^2_k(2c)<\eps^2(C)/3$.

Split $(0, C)$ into $(0,c_1),(c_1,c_2),\dots,(c_i,c'_k-\eps),(c'_k+\eps, c_{i+1}),\dots,(c_\ell, C)$ such that the size of any interval is less than $\delta$. For any $i\in[\binom{n}{2}]$, fix $\{A^i_j\}_{1\leq j\leq p_i}$ --- the sizes of non-trivial $k$-deeply connected components in $G_{c_in}$ ($p_i$ in total). For any $i\in[\ell]$, let $A^{(i)}_1, \dots, A^{(i)}_{p_{(i)}}$ denote the sizes of non-trivial $k$-deeply connected components in $G_{c_in}$. The following statement holds.

\begin{lemma}\label{lem:a_beta_commection}
    For any $\eps_0>0$, w.h.p.
    \[(1-\eps_0)(1-(A^{(j+1)}_1/n)^2-\dots-(A^{(j+1)}_p/n)^2)\leq1-\beta_k^2(2c_j) \leq (1+\eps_0)(1-(A^{(j-1)}_1/n)^2-\dots-(A^{(j-1)}_p/n)^2).\]
\end{lemma}
\begin{proof}
    Let $\mathbb{I}_i$ denote
    \[\mathbb{I}_i=\mathbb{I}[\textnormal{$e_i$ connects different $k$-deeply connected components in $G_{i-1}$}],\]
    Denote $\mathcal{P}'_i=\EE[\mathbb{I}_i|e_1,\dots,e_{i-1}]$. Note that for any $i$, $\mathcal{P}'_{i+1}<\mathcal{P}'_i+O(1/n^2)$:
    \[\mathcal{P}'_{i+1}-\mathcal{P}'_i=
    \EE[\mathbb{I}_{i+1}|e_1,\dots,e_i]-\EE[\mathbb{I}_i|e_1,\dots,e_{i-1}]<\frac{X}{\binom{n}{2}-m-1}-
    \frac{X}{\binom{n}{2}-m}=\frac{X}{\Theta(n^4)},\]
    where $X=\Theta(n^2)$ is the number of possible edge placements increasing $\rk G_{i-1}$ (it is important to mention that $X$ couldn't increase after adding $e_i$ due to the matroid properties of $\Mat^{(k)}$, in particular, \ref{eq:rank}).

    $\mathcal{P'}_{c_jn+1} = \frac{\binom{n}{2}-(A_1^{(j)})^2/2-\dots-(A_p^{(j)})^2/2+O(n)}{\binom{n}{2}-O(n)}=1-(A_1^{(j)}/n)^2-\dots-(A_{p_{(j)}}^{(j)}/n)^2+O(1/n)$, therefore
    \begin{multline*}
    1-\beta_k^2(2c_j)\leq
    (1+\eps_0/3)(\rk G_{c_jn} - \rk G_{c_{j-1}n})/(c_{j+1}-c_j)n\leq\\
    \leq(1+2\eps_0/3)\frac{\mathcal{P'}_{c_{j-1}n}+\dots+\mathcal{P'}_{c_jn-1}}{(c_{j+1}-c_j)n}=
    (1+2\eps_0/3)\frac{\sum\limits_{i=c_{j-1}n}^{c_jn-1}\left(1-\sum\limits_{\alpha=1}^{p_i} (A_\alpha^i)^2 + O(1/n)\right)}{(c_{j+1}-c_j)n}\overset{\mathrm{(i)}}{\leq}\\
    \overset{\mathrm{(i)}}{\leq}(1+2\eps_0/3)\frac{(c_{j+1}-c_j)n\left(1-\sum\limits_{\alpha=1}^{p_{(j-1)}} (A_\alpha^{(j-1)})^2 + O(1/n)\right)}{(c_{j+1}-c_j)n}\leq
    (1+\eps_0)(1-(A_1^{(j-1)}/n)^2-\dots-(A_{p_{(j-1)}}^{(j-1)}/n)^2).
    \end{multline*}

    The inequality (i) holds because for any $i$,
    \[\sum\limits_{j=1}^{p_i}(A_j^i)^2\leq\sum\limits_{j=1}^{p_{i+1}}(A_j^{i+1})^2.\]

    That proves one of the two bounds. On the other hand,
    \begin{multline*}
    1-(A_1^{(j)}/n)^2-\dots-(A_{p_{(j)}}^{(j)}/n)^2=\mathcal{P}'_{c_jn+1}+O(1/n)\leq \frac{\mathcal{P}'_{c_{j-1}n}+\dots+\mathcal{P}'_{c_jn-1}}{(c_j-c_{j-1})n}+O(1/n)\overset{\mathrm{(ii)}}{\leq}\\
    \overset{\mathrm{(ii)}}{\leq} (1+\eps_0/2)\frac{\rk G_{c_jn} - \rk G_{c_{j-1}n}}{(c_j-c_{j-1})n}\leq(1+\eps_0)(1-\beta^2_k(2c_{j-1}n)),
    \end{multline*}
\end{proof}

but we should pay extra attention to the inequality marked by (ii):

\begin{lemma}
    Fix $b>a>0$. Then, the following holds:
    \[\frac{\mathcal{P}'_{bn}+\dots+\mathcal{P}'_{an-1}}{(b-a)n}-\frac{\rk G_{bn} - \rk G_{an}}{(b-a)n}\overset{\mathrm{P}}{\to}0.\]
\end{lemma}
\begin{proof}
    By definition, $\mathbb{I}_i=\rk G_i-\rk G_{i-1}$, then the statement can be equivalently expressed as
    \[\mathfrak{S}:=\sum\limits_{i=bn}^{an-1}(\mathcal{P}'_i-\mathbb{I}_i)=o_p(n).~~\left(\textnormal{i.e. } \frac{\mathfrak{S}}{n}\overset{\mathrm{P}}{\to}0\right)\]

    Note that $\displaystyle\EE[\mathcal{P}'_i] = \EE\big[\EE[\mathbb{I}_i|e_1,\dots,e_{i-1}]\big]=\EE[\mathbb{I}_i]$, which means that $\EE \mathfrak{S}=0$.

    Now, estimate $\EE \mathfrak{S}^2$:
    \begin{enumerate}
        \item $\EE (\mathcal{P}'_i-\mathbb{I}_i)^2\leq1$ since $|\mathcal{P}'_i-\mathbb{I}_i|\leq1$ a.s.
        \item Consider $(\mathcal{P}'_i-\mathbb{I}_i)(\mathcal{P}'_j-\mathbb{I}_j)$ for $i\neq j$. Note that $\mathbb{I}_i, \mathcal{P}'_i$ are $(e_1,\dots,e_{i})$-measurable. Let $i<j$, then
        \begin{multline*}
            \EE\big[(\mathcal{P}'_i-\mathbb{I}_i)(\mathcal{P}'_j-\mathbb{I}_j)\big|e_1,\dots e_i\big]=(\mathcal{P}'_i-\mathbb{I}_i)\EE\big[\mathcal{P}'_j-\mathbb{I}_j\big|e_1,\dots e_i\big] = \\
            =(\mathcal{P}'_i-\mathbb{I}_i)\Big(\EE\big[\mathbb{I}_j\big|e_1,\dots e_i\big]-\EE\big[\mathbb{I}_j\big|e_1,\dots e_i\big]\Big)=0,
        \end{multline*}
        and, consequently, $\EE[(\mathcal{P}'_i-\mathbb{I}_i)(\mathcal{P}'_j-\mathbb{I}_j)]=0$.
    \end{enumerate}
    Therefore $\EE \mathfrak{S}^2=O(n)$ and $\mathfrak{S}$ is indeed $o_p(n)$.
\end{proof}

Equivalently to the statement of Lemma \ref{lem:a_beta_commection}, for any $\eps_0>0$ w.h.p.
\[(1+\eps_0)(1-\beta_k^2(2c_{i-1}))\geq1-(A^i_1/n)^2-\dots-(A^i_p/n)^2\geq(1-\eps_0)(1-\beta_k^2(2c_{i+1})).\]

Let $t\in[(c'_k+\eps)n, Cn)$ denote the first moment such that $G_t$ contains at least two big ($\geq 2$ vertices) $k$-deeply connected components, if such moment happens. Let $t\in(c_j,c_{j+1})$ for some $j$ and let $A,B$ denote the number of vertices in these components. Then similarly, for any $\eps_0>0$, w.h.p.
\[(1+\eps_0)(1-\beta_k^2(2c_{j-1}))\geq1-(A/n)^2-(B/n)^2\geq(1-\eps_0)(1-\beta_k^2(2c_{j+1})).\]

On the other hand, $G_{t-1}$ only contains one component of size either $A$ or $B$. Let $A$ be that size. Similarly, w.h.p.
\[(1+\eps_0)(1-\beta_k^2(2c_{j-1}))\geq1-(A/n)^2\geq1-(A/n)^2-(B/n)^2>(1-\eps_0)(1-\beta_k^2(2c_{j+1})),\]
from which it is apparent that
\[(B/n)^2\leq2\eps_0+2\eps^2(C)/3.\]

Let $\eps_0=\eps^2(C)/12$, then w.h.p. $(B/n)^2\leq5\eps^2(C)/6$, but $(B/n)^2>\eps^2(C)$ since density of the subgraph induced by the vertices of this component in $G_{Cn}$ is at least $3$. Therefore, for any $C>0$ w.h.p. for any $t<Cn$, there is at most one big $k$-deeply connected component, and since
\[(1+\eps_0)(1-\beta_k^2(2c_{i-1}))\geq1-(A^i_1/n)^2\geq(1-\eps_0)(1-\beta_k^2(2c_{i+1})),\]
we obtain that  w.h.p. $A^i_1\sim\beta_k(2c_i)n$, or, equivalently, for any $c>c'_k$ there exists a single giant $k$-deeply connected component of size $\beta_k(2c)n(1+o(1))$.

\section{Conclusion and Prospects}

This paper answers a question raised by A. Frieze and T. Johansson in \cite{FRIEZE2017}: is that true that the rank of the $(k+1)$-core is equal to the $k$ times its size? The answer is yes and this allowed for the discovery of the completely new structure, the \textit{giant $k$-deeply connected component}. Recall that if $c>2c'_k$, its size in in $G(n,c/n)$ converges to $\beta_k(c)$, the largest root of the equation
\[\beta=\Pr[\Pois(\beta c)\geq k].\]
But if we consider $\gamma_k=\Pr[\Pois(\beta_k(c) c)\geq k+1]$, then it turns out to be the asymptotic size of $(k+1)$-core. This strange connection arises several questions yet to be answered:
\begin{enumerate}
    \item Is that true that the $(k+1)$-core lies entirely in the giant $k$-deeply connected component w.h.p.?
    \item Is that true that the $(k+1)$-core is $k$-deeply connected w.h.p.?
    \item Let $\tau$ denote the moment of emergence of the giant $k$-deeply connected component in the process $G_1\subsetneq\dots\subsetneq G_{\binom{n}{2}}$. Is that true that $\tau$ is w.h.p. also the first moment when the $(k+1)$ core has $v$ vertices and at least $k(v-1)$ edges?
    \item Does the Central Limit Theory hold for the size of the giant $k$-deeply connected component in uniform and/or binomial models?
\end{enumerate}

\newpage
\bibliographystyle{plainurl}
\bibliography{refs}
\newpage

\appendix
\section{Proof of Lemma \ref{lem:small_dense}}
    We will prove a seemingly slightly weaker result, proving that there are not any subgraphs on $\delta n$ vertices or less with density at least $2+2\eps$ (opposed to $2+\eps$ in the original statement) for some $\delta > 0$. We also will consider binomial model $G(n,c/n)$. But upon the closer inspection it becomes clear that these statements are equivalent, because the proof works for any positive constants, for example, for $\eps/2$ and $1000c$. Using the fact that "having dense small subgraphs" is a monotone property, we can easily bound from below the amount of edges in $G(n, 1000c/n)$  as $cn$ and conclude that the same statement holds for $G(n, cn)$.

    For any $\alpha$, examine the probability
    \[\mathcal{P}_\alpha = \Pr[\textnormal{there exists a subgraph on $\alpha$ vertices with the density at least $2+2\eps$}].\]
    For any $\alpha$, $\mathcal{P}_\alpha$ is not larger then the expected number of dense subgraphs on $\alpha$ vertices, which can be bounded by
    \[\binom{n}{\alpha}\binom{\binom{\alpha}{2}}{\lceil1+\eps\rceil\alpha}p^{\lceil1+\eps\rceil\alpha}\leq\binom{n}{\alpha}\frac{\Gamma(\alpha^2/2+1)}{\Gamma((1+\eps)\alpha+2)\Gamma(\alpha^2/2-(1+\eps)\alpha)}p^{(1+\eps)\alpha}\eqdef \mathcal{P}'_\alpha,\]
    where the inequality holds for sufficiently large $\alpha$ ($\alpha> f(\eps)$ for some $f$). Note that $\mathcal{P}_\alpha=o(1)$ for any fixed $\alpha$ since it can be bounded from above by $\frac{c^{(1+\eps)\alpha}}{n^{\eps\alpha}}\cdot g(\alpha)$ for some $g$. Additionally, $\mathcal{P}'_{\lceil2/\eps\rceil}=O(n^{-2})$ for any fixed $c,\eps$.

    Now examine $D_\alpha:=\frac{\mathcal{P}'_\alpha}{\mathcal{P}'_{\alpha-1}}$. As long as $D_\alpha<1$, $\mathcal{P}'_\alpha$ is decreasing, which means that if $D_\alpha<1$ for any $\alpha\in[\lceil2/\eps\rceil, \delta n]$, then
    \[\Pr[\textnormal{exists a dense subgraph of size at most $\delta n$}]\leq\sum\limits_{i\leq\delta n}\mathcal{P}_i\leq\mathcal{P}'_1+\dots+\mathcal{P}'_{\lceil2/\eps\rceil} \cdot (\delta n - \lceil2/\eps\rceil + 1) = o(1).\]

    Now let's estimate $D_\alpha$:
    \begin{multline*}
        D_\alpha=\frac{\mathcal{P}'_\alpha}{\mathcal{P}'_{\alpha-1}}\leq\frac{n-\alpha}{1+\alpha}\cdot p^{1+\eps}\cdot\frac{\left(\frac{(\alpha+1)^2}{2}\right)^{\alpha+1/2}}{((1+\eps)\alpha)^{1+\eps}(\alpha^2/2-(1+\eps))\alpha)^{\alpha-1/2-\eps}}\leq\\
        \leq\frac{c^{1+\eps}}{n^\eps}\frac{\left(\frac{(\alpha+1)^2}{2}\right)^{1+\eps}}{(1+\alpha)((1+\eps)\alpha)^{1+\eps}}\left(\frac{\alpha^2+2\alpha+1}{\alpha^2-(2+2\eps)\alpha}\right)^{\alpha-1/2-\eps}\leq\frac{\alpha^\eps}{n^\eps}h(\eps,c),
    \end{multline*}
    where $h$ is some positive function. As it easily seen, if $\delta=\sqrt[\eps]{1/h(\eps,c)}$, then $\sum\limits_{\alpha\leq\delta n}\mathcal{P}_\alpha=o(1).$

\section{Proof of Lemma \ref{lem:core_rank}}
    If the $k+1$-core did not get new vertices upon drawing an additional edge $e$, $kV_1-\Delta\rk C_{k+1}\leq0$.

    Otherwise, there are two different cases: either $\rk G$ increased after adding $e$ or it didn't. Consider the first case.
    Note the following:
    \begin{enumerate}
        \item Since these $V_1$ vertices are now a part of the $(k+1)$-core, their degrees in the graph induced by the core are at least $k+1$: $E'+2E_1\geq(k+1)V_1$;
        \item the rank of the kernel changed by exactly $E'+E_1$ (by Lemma \ref{lem:rk}, all of these edges lie in the largest independent set of edges in the graph).
    \end{enumerate}
    Hence, $E_1-V_1\geq kV_1-\Delta\rk C_{k+1}$.

    Now, consider the second case: $e$ is dependent of $E(G)$. Lemma \ref{lem:connection} implies that both vertices of $e$ lie in the same $k$-deeply connected component. We will prove that after adding $e$, $k|C_{k+1}|-\rk C_{k+1}$ does not change.
    \begin{enumerate}
        \item If both vertices of $e$ were in $C_{k+1}$, then the fact is trivial: size of the core did not change, neither did the rank.
        \item If one of the vertices of $e$ is not in the $(k+1)$-core, consider the normal representation of the $k$-deeply connected component of $e$. Add $e$ to this representation without rearranging the layers. Note that at least one of the ends is not in the layer zero, otherwise, both ends of $e$ would have been in the core. Note, that after adding $e$, its ends with all the children (vertices that can be reached from the ends only by going to previous layers) together with the first layer form a $(k+1)$-core, therefore, are either in the core already or among these $V_1$ vertices added to the core. If no more vertices were added, then every added vertex also contains $k$ edges lying in the maximum independent set in $\Mat^k$, and $k|C_{k+1}|-\rk C_{k+1}$ won't change. If that is not the case, consider all the $V_1$ vertex added to the core and split it into 2 parts: ends of $e$ together with their children (part 1) and the rest (part 2). Now consider only vertices of part 1 connected to part 2 together with their children, all the vertices of part 2 and the core. Every vertex in what is left of part 1 has $k$ edges going down in the normal representation and additionally at least one edge going up and/or at least one edge going to part 2. All the edges not entirely in part 1 weren't erased, which means that vertices in part 2 also have degree at least $k+1$. Therefore, these vertices were already in a $(k+1)$-core, which leads to a contradiction. Then, indeed, no more vertices were added besides the ends of $e$ and their children
        \item If both ends of $e$ aren't in the core, the proof is similar to the previous case.
    \end{enumerate}
    Now, we know that if $kV_1-\Delta\rk C_{k+1}>0$, then
    \[E_1-V_1\geq kV_1-\Delta\rk C_{k+1}.\]

\section{Proof of Lemma \ref{lem:beta}}
    First of all, note that
    \[(\pi_k(x))'_x=\left(\sum\limits_{i=k}^{\infty}\frac{x^ie^{-x}}{i!}\right)'_x=\pi_{k-1}(x)-\pi_k(x).\]
    Remember that $\lambda_k(2c)=2c\cdot\pi_{k-1}(\lambda_k(2c))$. Later in the proof, we substitute $\lambda_k(2c)$ with $\lambda_k$. By differentiating the equation we get
    \[\lambda'_k=2\pi_{k-1}(\lambda_k)+2c\cdot(\pi_{k-1}(\lambda_k))'=2\pi_{k-1}(\lambda_k)+2c\cdot\lambda'_k(\pi_{k-2}(\lambda_k)-\pi_{k-1}(\lambda_k)),\]
    or, equivalently,
    \[\lambda'_k=\frac{2\pi_{k-1}(\lambda_k)}{1-2c\left(\pi_{k-2}(\lambda_k)-\pi_{k-1}(\lambda_k)\right)}.\]
    As we know from Corollary \ref{cor:r_def},
    \begin{multline*}
        \frac{\partial}{\partial c}r_k(c)=1-\left(\frac{\lambda_{k+1}\pi_k(\lambda_{k+1})}{2}\right)'+k(\pi_{k+1}(\lambda_{k+1}))'=1-\left(\frac{\lambda^2_{k+1}}{4c}\right)'+k(\pi_{k+1}(\lambda_{k+1}))'=\\
        1-\frac{\lambda_{k+1}\lambda'_{k+1}}{2c}+\frac{\lambda^2_{k+1}}{4c^2}+k\lambda'_{k+1}(\pi_{k}(\lambda_{k+1})-\pi_{k+1}(\lambda_{k+1})).
    \end{multline*}
    We want to prove that $\frac{\partial}{\partial c}r_k(c)=1-\beta_k(2c)^2$, or, equivalently, $\lambda_{k+1}^2/4c^2=\beta_k(2c)^2=1-\frac{\partial}{\partial c}r_k(c)$. Then,
    \[1-\frac{\partial}{\partial c}r_k(c)=\frac{\lambda_{k+1}\lambda'_{k+1}}{2c}-\frac{\lambda^2_{k+1}}{4c^2}-k\lambda'_{k+1}(\pi_{k}(\lambda_{k+1})-\pi_{k+1}(\lambda_{k+1})).\]

    Let's attempt to substitute $1-\frac{\partial}{\partial c}r_k(c)$ with $\beta_k(2c)^2$:
    \begin{align*}
        \frac{\lambda_{k+1}^2}{4c^2} &\vee  \frac{\lambda_{k+1}\lambda'_{k+1}}{2c}-\frac{\lambda^2_{k+1}}{4c^2}-k\lambda'_{k+1}(\pi_{k}(\lambda_{k+1})-\pi_{k+1}(\lambda_{k+1})); \\
        \frac{\lambda_{k+1}^2}{2c^2} &\vee \frac{\lambda_{k+1}\lambda'_{k+1}}{2c} -k\lambda'_{k+1}(\pi_{k}(\lambda_{k+1})-\pi_{k+1}(\lambda_{k+1})); \\
        2\pi^2_k(\lambda_{k+1}) &\vee \pi_k(\lambda_{k+1})\lambda'_{k+1} -k\lambda'_{k+1}(\pi_{k}(\lambda_{k+1})-\pi_{k+1}(\lambda_{k+1})); \\
        2\pi^2_k(\lambda_{k+1}) &\vee \lambda'_{k+1}\left(\pi_k(\lambda_{k+1}) -k(\pi_{k}(\lambda_{k+1})-\pi_{k+1}(\lambda_{k+1}))\right); \\
        \frac{2\pi^2_k(\lambda_{k+1})}{k\pi_{k+1}(\lambda_{k+1})-(k-1)\pi_k(\lambda_k)} &\vee \lambda'_{k+1}; \\
        \frac{2\pi^2_k(\lambda_{k+1})}{k\pi_{k+1}(\lambda_{k+1})-(k-1)\pi_k(\lambda_k)} &\vee \frac{2\pi_{k}(\lambda_{k+1})}{1-2c\left(\pi_{k-1}(\lambda_{k+1})-\pi_{k}(\lambda_{k+1})\right)}; \\
        \pi_{k}(\lambda_{k+1})\left(1-2c\left(\pi_{k-1}(\lambda_{k+1})-\pi_{k}(\lambda_{k+1})\right)\right) &\vee k\pi_{k+1}(\lambda_{k+1})-(k-1)\pi_k(\lambda_k); \\
        2c\pi_{k}(\lambda_{k+1})\left(\pi_{k-1}(\lambda_{k+1})-\pi_{k}(\lambda_{k+1})\right) &\vee k\pi_k(\lambda_{k+1})-k\pi_{k+1}(\lambda_k); \\
        \lambda_{k+1}\left(\pi_{k-1}(\lambda_{k+1})-\pi_{k}(\lambda_{k+1})\right) &\vee k(\pi_k(\lambda_{k+1})-\pi_{k+1}(\lambda_k)); \\
        \lambda_{k+1}\frac{\lambda_{k+1}^{k-1}e^{-\lambda_{k+1}}}{(k-1)!} &\vee k\frac{\lambda_{k+1}^{k}e^{-\lambda_{k+1}}}{k!}; \\
        \frac{\lambda_{k+1}^{k}e^{-\lambda_{k+1}}}{(k-1)!} &= \frac{\lambda_{k+1}^{k}e^{-\lambda_{k+1}}}{(k-1)!}; \\
    \end{align*}
    Indeed, $1-\frac{\partial}{\partial c}r_k(c)=\beta_k(2c)^2$.

    We know that $\lambda_k(c)=c\cdot\pi_{k-1}(\lambda_k(c))$ and that $\lambda_k(c)$ is the largest root of this equation. By rearranging the equation, we get the following:
    \[\frac{\lambda_k(c)}{c}=\pi_{k-1}\left(\frac{\lambda_k(c)}{c}\cdot c\right),\]
    from which it is obvious that $\beta_k(c)$ is in fact the largest root of the equation $\beta=\pi_k\left(\beta c\right)$.

\end{document}